\documentclass[11pt]{amsart} \usepackage{verbatim}
\usepackage{rotating} 
\usepackage{amssymb} \usepackage{color} \usepackage{amsmath}

\newtheorem{theorem}{Theorem}

\newtheorem{proposition}[theorem]{Proposition}
\newtheorem{prop}[theorem]{Proposition}

\newtheorem{lemma}[theorem]{Lemma}

\newtheorem{question}[theorem]{Question}
\newtheorem*{definition}{Definition}

\theoremstyle{plain}
\numberwithin{equation}{theorem}

\theoremstyle{remark}

\newtheorem{remark}{Remark}

\newcommand{\C}{{\mathbb C}}

\newcommand{\Q}{{\mathbb Q}}

\newcommand{\cJ}{{\mathcal J}}

\newcommand{\cV}{{\mathcal V}}

\newcommand{\cS}{{\mathcal S}}

\newcommand{\Kbar}{\overline K}
\newcommand{\kbar}{\overline{k}}

\newcommand{\Qbar}{\overline{\Q}}

\DeclareMathOperator{\bN}{\mathbb{N}}

\newcommand{\bA}{{\mathbb A}}
\newcommand{\bC}{{\mathbb C}}
\newcommand{\bF}{{\mathbb F}}
\newcommand{\bG}{{\mathbb G}}

\newcommand{\bP}{{\mathbb P}}
\newcommand{\bQ}{{\mathbb Q}}

\newcommand{\bZ}{{\mathbb Z}}

\newcommand{\lra}{\longrightarrow}

\newcommand{\cT}{\mathcal{T}}

\newcommand{\bQb}{\overline \bQ}

\newcommand{\hhat}{{\widehat h}}

\newcommand{\PrePer}{\operatorname{PrePer}}

\newenvironment{parts}[0]{%
  \begin{list}{}%
    {\setlength{\itemindent}{0pt}
     \setlength{\labelwidth}{1.5\parindent}
     \setlength{\labelsep}{.5\parindent}
     \setlength{\leftmargin}{2\parindent}
     \setlength{\itemsep}{0pt}
     }%
   }%
  {\end{list}}
\newcommand{\Part}[1]{\item[\upshape#1]}

\begin{document}

\title{Greatest common divisors of iterates of polynomials}

\author{L.-C.~Hsia}
\address{
Liang-Chung Hsia\\
Department of Mathematics\\
National Taiwan Normal University\\
Taipei, Taiwan, ROC
}
\email{hsia@math.ntnu.edu.tw}

\author{T.~J.~Tucker}
\address{
Thomas Tucker\\
Department of Mathematics\\
University of Rochester\\
Rochester, NY 14627\\
USA
}
\email{ttucker@math.rochester.edu}

\thanks{The first author was partially supported by MOST Grant
  104-2115-M-003-004-MY2. The second author was partially supported
  by NSF Grant DMS-0101636.}

\begin{abstract}
  Following work of Bugeaud, Corvaja, and Zannier for integers, Ailon
  and Rudnick prove that for any multiplicatively independent
  polynomials, $a, b \in \bC[x]$, there is a polynomial $h$ such that
  for all $n$, we have
  \[ \gcd(a^n - 1, b^n - 1) \mid h\]
  We  prove a compositional analog of this theorem, namely that if
  $f, g \in \bC[x]$ are nonconstant compositionally independent
  polynomials and $c(x) \in \bC[x]$, then there are at most finitely
  many $\lambda$ with the property that there is an $n$ such that $(x
  - \lambda)$ divides  $\gcd(f^{\circ n}(x) - c(x), g^{\circ n}(x) - c(x))$.
  \end{abstract}

\maketitle

In the paper~\cite{BCZ}, Bugeaud, Corvaja, and Zannier obtained an
upper bound for the greatest common divisors among two families of integer sequences. More precisely, let $a$ and $b$ be two positive integers that are multiplicatively independent and let $\epsilon > 0$ be given. Then for all $n$, we have $ \gcd(a^n - 1, b^n -1) \ll_\epsilon \exp(\epsilon n) $ where the implied constant is  independent of $n$.

Since Bugeaud, Corvaja, and Zannier's paper appeared, there have been
many extensions and generalizations of their results,
see for example \cite{AR, CZ05, Luca, Sil04, SilverGCD}. In the setting over function field of characteristic zero, Ailon and Rudnick~\cite{AR} obtained a stronger upper bound.  They showed that for two multiplicatively independent nonconstant polynomials $a, b \in \bC[x]$,  there is a polynomial $h\in \bC[x]$, depending on $a$ and $b$ such that $\gcd(a^n - 1, b^n - 1) \mid h$ for all positive integer $n$.
We note here that the result of Ailon and Rudnick also holds when one takes the
greatest common divisors of $a^m - 1$ and $b^n - 1$ across all pairs
of positive integers $m$ and $n$ (not merely those where $m = n$).

Instead of taking multiplicative powers of polynomials, one can consider iterated compositions of polynomials and look for an upper bound on the degrees of the greatest common divisors among two such sequences of polynomials as asked by A.~Ostafe in~\cite[Problem 4.2]{Ostafe}.
In this paper, we prove a compositional analog of theorem of Ailon and
Rudnick described above.

In the following, for a polynomial $q$, we let $q^{\circ n}$ denote
the composition of $q$ with itself $n$ times. To state our theorem
precisely, we need a definition of compositional independence.
\begin{definition}
  We say two polynomials $f$ and $g$ are {\em compositionally
    independent} if the semigroup generated by $f$ and $g$ under
  composition is isomorphic to the free semigroup with two generators.
  This is equivalent to the property that whenever $i_1, \dots, i_s$,
  $j_1, \dots, j_s$, $\ell_1, \dots, \ell_t$, $m_1, \dots m_t$ are
  positive integers such that
\[ f^{\circ i_1} \circ g^{\circ j_1} \circ \cdots \circ  f^{\circ i_s}
\circ g^{\circ j_s} =  f^{\circ \ell_1} \circ g^{\circ m_1} \circ \cdots
\circ  f^{\circ \ell_t} \circ g^{\circ m_t}, \]
we must have $s=t$, and $i_k= \ell_k$, $j_k = m_k$ for $k= 1,
\dots, s$.
\end{definition}

Under the compositional independence condition, our first result is the finiteness of the irreducible factors of $\gcd(f^{\circ m}(x) - c(x), g^{\circ n}(x) - c(x))$ where $f, g$ and $c$ are polynomials with complex coefficients.  More precisely,  we have the following theorem which answers Ostafe's question.

\begin{theorem}\label{main thm}
  Let $f(x)$ and $g(x)$ be two compositionally independent polynomials
  in $\C[x]$, at least one of which has degree greater than one.
  Suppose that $c(x)$ is not a compositional power of $f$ or $g$.
  Then there are at most finitely many $\lambda \in \bC$ such that
\[ (x - \lambda) | \gcd(f^{\circ m}(x) - c(x), g^{\circ n}(x) -
c(x)) \]
for some positive integers $m,n$.
\end{theorem}

The restriction on the degrees of the two polynomials $f$ and $g$ in
Theorem~\ref{main thm} is necessary. As the examples at the beginning
of Section~\ref{linear} demonstrate that Theorem~\ref{main thm} must
be modified when $f$ and $g$ are both linear. If we restrict to the
case $m=n$ in Theorem~\ref{main thm}, then we still obtain a
finiteness result when the two polynomials $f$ and $g$ are both
linear.

\begin{theorem}\label{linear-thm}
  Let $f$ and $g$ be two compositionally independent linear polynomials and let $c$ be any polynomial.  Then there is a polynomial $h \in \bC[x]$ such that
\[ \gcd(f^{\circ n}(x) - c(x), g^{\circ n}(x) -
c(x))  \mid h  \]
for all positive integers $n$.
\end{theorem}

Putting Theorem \ref{linear-thm} together with Theorem~\ref{main thm} under the condition that the composition power $m=n$, then for {\em any} polynomials $c(x)$ we have the same conclusion.

\begin{theorem}\label{just-n}
Let $f$ and $g$ be two compositionally independent polynomials. Then there are at most finitely many $\lambda \in \bC$ such that
\[ (x - \lambda) | \gcd(f^{\circ n}(x) - c(x), g^{\circ n}(x) -
c(x)) \]
for some positive integer $n$.
\end{theorem}

We note that Theorem~\ref{linear-thm} is a compositional analogue of Ailon and Rudnick's result for linear polynomials.
To obtain a theorem that is parallel to their result 
for non-linear polynomials, we need a bound for the multiplicity of each irreducible factor that divides the greatest common divisors. In general, one can not expect such a bound  exists. For instance, take  $f(x) = x^3 + x^2, g(x) = x^3 + 5 x^2$ and $c=0.$ Then, for any positive integer
$n$, we have
\[ x^{2^n} | \gcd(f^{\circ n}(x), g^{\circ n}(x)) \]
Hence, in this case there does not exist a polynomial $h$ divisible by all the greatest common divisors of the sequences in question.
To get control on the bound of the multiplicities of irreducible factors dividing the greatest common multiples, we need one extra condition.
\begin{definition}
We say that $c \in \bC$ is in a ramified cycle of a polynomial $q$
if there is a positive integer $i$ such that $q^{\circ i}(c) = c$ and
$(q^{\circ i})'(c) = 0$.
\end{definition}

Once we exclude this sort of possibility, we are able to show that there exists a polynomial that is divisible by all the greatest common divisors of the compositional sequences formed by $f$ and $g$.

\begin{theorem}\label{gcdthm}
  Let $f(x)$ and $g(x)$ be two compositionally independent polynomials
  of degrees greater than one in $\C[x]$.  Suppose that $c(x)$ is not a
  compositional power of $f$ or $g$.  Supposer furthermore that $c(x)$
  is not equal to a constant $c$ that is in a ramified cycle of both
  $f$ and $g$.  Then there is a polynomial $h \in \bC[x]$ such that
\[ \gcd\left(f^{\circ m}(x) - c(x), g^{\circ n}(x) -
c(x)\right)  \mid h  \]
for all positive integers $m,n$.

\end{theorem}

\begin{remark}
  (1) In the situation considered by Ailon and Rudnick, the number $1$ is not in a ramified cycle of any powering map. In fact, any non-zero polynomial $c(x)$ is not in a ramified cycle of any powering map. \smallskip\\
  (2) For a given pair of multiplicatively independent polynomials $a$
  and $b$, one might ask whether there exists a bound on the degrees of
  $\gcd(a(x)^m - c(x), b(x)^n - c(x))$ that is independent not
  only of $m$ and $n$ but also of the choice of non-zero polynomial
  $c(x)$.
\end{remark}

We give a brief description of the organization of our paper and explain the ideas of the proofs.
In Section \ref{prelim}, we set up notations and provide some background
about canonical height functions associated to rational maps on the projective line over a global field.  After the preliminaries in Section~\ref{prelim}, we begin to prove our results.

We prove Theorem \ref{main thm} in Section \ref{nonlin}. The proof is
split into two parts.  We first treat the case where neither $f$ nor
$g$ is linear. This is done in Proposition~\ref{nonlinear}.  As
additional ingredient is required for the case where one of $f$ and
$g$ is linear; we treat this case separately in
Proposition~\ref{onelinear}.  Then Theorem~\ref{main thm} is just the
combination of these two propositions.  We sketch the proof of
Proposition~\ref{nonlinear} here.  Assuming that the set of $\lambda$
that are roots of $\gcd(f^{\circ m}(x) - c(x), g^{\circ n}(x) - c(x))$
is infinite as $m, n$ run through all positive integers. Then these
numbers have the property that the canonical heights
$\hhat_f(\lambda)$ and $\hhat_g(\lambda)$ both converge to zero (see
Lemma~\ref{special}).  Applying equidistribution theorems in
arithmetic dynamics, following the pattern of \cite{Bogo, BD1, GHT1},
we conclude that both polynomials $f$ and $g$ have the same Julia set
in the complex plane. Then the work of Baker/Er\"{e}menko and
Schmidt/Steinmetz \cite{Baker-Eremenko, SS} shows that a compositional
relation between $f$ and $g$ exists. Thus we get a contradiction to
the assumption that $f$ and $g$ are compositionally independent and
finish the proof.

Section~\ref{linear} is devoted to the proof of Theorem~\ref{linear-thm} and
Theorem~\ref{just-n}.
The proof of Theorem~\ref{linear-thm} is quite different, as the tools used to prove Theorem~\ref{main thm} are no longer applicable to the case where both polynomials $f$ and $g$ are linear. The proof for this case relies heavily on diophantine methods, in particular an application of results from \cite{CZ05}, Roth's theorem, and a lemma of Siegel.
These results are used to prove the case where everything is
defined over $\bQb$, in Proposition~\ref{algebraic}. The general case of Theorem
\ref{linear-thm} then follows via specialization. Theorem~\ref{just-n} follows easily by combining Theorem~\ref{main thm} and Theorem~\ref{linear-thm}.

We prove Theorem~\ref{gcdthm} in Section~\ref{gcd}.  It is sufficient
to bound the multiplicities of the roots of
$\gcd(f^{\circ m}(x) - c(x), g^{\circ n}(x) - c(x))$ in
Theorem~\ref{main thm} provided that $c(x)$ is not a constant in a
ramified cycle of both $f$ and $g$. The analysis on the bound of the
multiplicity used here is similar to those used in~\cite[Lemma~3.4]{MorSil2}. We
provide such a bound in Lemma~\ref{close}. Then, Theorem~\ref{gcdthm}
follows from Theorem~\ref{main thm} coupled with Lemma~\ref{close}.
Finally, we end this paper by raising several questions for further study in
Section~\ref{sec:question}.

\vskip1.5mm
\noindent {\bf Acknowledgments.}  We would like to thank Alina Ostafe,
Juan Rivera-Letelier, Umberto Zannier, Shouwu Zhang, and Mike Zieve for helpful
conversations. The first named author would like to thank his coauthor for his hospitality during the visit to the Math. Department of University of Rochester in the summer of 2014 when this project was initiated.

\section{Preliminaries} \label{prelim}

In this section, we set up some notations and recall facts from the theory of height functions that will be used in this paper.


Let $K$ be a field of characteristic $0$ equipped with a set of inequivalent
absolute values (places) $\Omega_K$, normalized so that the product
formula holds. More precisely, for each $v\in\Omega_K$ there exists a
positive integer $N_v$ such that for all $\alpha\in K^{\ast}$ we have
$\prod_{v\in \Omega}\,|\alpha|_v^{N_v} = 1$ where for $v\in \Omega_K$, the
corresponding absolute value is denoted by $|\cdot |_v$.  Examples of
product formula fields (or \emph{global fields}) are number fields and function fields of
projective varieties which are regular in codimension 1 over another field $k$ (see
\cite[\S~2.3]{lang} or \cite[\S~1.4.6]{BG}).

We let $\C_v$ be the completion of an algebraic closure of $K_v$, a completion of $K$ with respect to  $|\cdot |_v$. When $v$ is an archimedean valuation, then $\C_v=\C$. We fix an extension of $|\cdot |_v$ to an absolute value of $\C_v$ which by abuse of notation, we still denote it by $|\cdot|_v$.

If $K$ is a number field,  we let $\Omega_K$ be the set of all absolute values  of $K$ which extend the (usual) absolute values of $\bQ$. For each $v\in\Omega_K$, we let $v_0$ denote the (unique) absolute value of $\bQ$ such that $v|_\bQ=v_0$ and we let $N_v:=[K_v:\bQ_{v_0}]$.
If $K$ is a function field of a projective normal variety $\cV$ defined over a field $k$, then $\Omega_K$ is the set of all absolute values on $K$ associated to the irreducible divisors of $\cV$. Then there exist positive integers $N_v$ (for each $v\in\Omega_K$) such that $\prod_{v\in\Omega_K}|x|_v^{N_v}=1$ for each nonzero $x\in K$.
(see \cite{dio-geo,Serre-Mordell_Weil} for more details).

Let $L$ be a finite extension of $K$, and let $\Omega_L$ be the set of all absolute values of $K$ which extend the absolute values in $\Omega_K$. For each $w\in \Omega_L$ extending some $v\in\Omega_K$ and let $N_w:=N_v\cdot [L_w:K_v]$. The (naive) Weil height of any point $x\in L$ is defined as
$$h(x)=\frac{1}{[L:K]}\sum_{w\in\Omega_L}N_w\cdot \log\max\{1, |x|_w\}.$$
To ease the notation, we set
$\|x\|_v := |x|_v^{N_v}$ for $x\in K$.

Let $f\in K(x)$ be any rational map  of degree $d\ge 2$.  Then the \emph{global canonical height} $\hhat_f(x)$ of $x\in\Kbar$ associated to $f$ is given by the limit
$$\hhat_f(x)=\lim_{n\to\infty} \frac{h(f^n(x))}{d^n}$$
(see \cite{CS} for details).
In addition, Call and Silverman proved that the global canonical height decomposes as a sum of the local canonical heights, i.e.
\begin{equation}
\label{decomposition for the global canonical height}
\hhat_f(x)=\frac{1}{[K(x):K]}\sum_{\sigma:K\to \Kbar}\sum_{v\in\Omega_K}N_v \hhat_{f,v}\left(x^\sigma\right),
\end{equation}
where for each $v\in \Omega_{K}$ the function $\hhat_{f,v}$ is the
local canonical height associated to $f.$ For the existence and
functorial property of the local canonical height see~\cite[Theorem~2.1]{CS}.

The following facts about height functions are well-known.

\begin{proposition}
\label{prop:heightproperty}
Let $f \in K(x)$ be a rational function of degree $d\ge2$ defined
over $K$.  There are constants $c_1$,~$c_2$,~$c_3$, and $c_4$,
depending only on $d$, such that the following estimates hold for all
$x \in \Kbar$.
\begin{parts}
\Part{(a)}
$\bigl|h(f(x)) - d h(x)\bigr| \le c_1 h(f) + c_2$.
\Part{(b)}
$ \bigl| \hhat_f(x) - h(x)\bigr| \le c_3 h(f) + c_4$.
\Part{(c)}
$\hhat_f(f(x)) = d \hhat_f(x)$.
\Part{(d)} If $K$ is a number field then
$x \in \PrePer(f)$ if and only if $\hhat_f(x) = 0$.
\end{parts}
Here, $h(f)$ is the height of the polynomial $f$, see for example~\cite[Sect.~1.6]{BG} for the definition of $h(f).$
\end{proposition}
\begin{proof}
See, for example, \cite[\S\S B.2,B.4]{hindry-silverman00}
or \cite[\S3.4]{silverman_book07}.
\end{proof}

We use the following lemma (see also \cite{CS, Ingram}
for more general techniques along these lines).

\begin{lemma}\label{special}
  Let $K$ be a global field. Let $(\lambda_n)_{n=1}^\infty$ be a  sequence in $\Kbar$ satisfying   $f^{\circ n}(\lambda_n) = c(\lambda_n)$ for all $n$, where
  $f, c \in K[x]$ and $\deg f > 1$.  Then
\[ \lim_{n \to \infty} \hhat_f(\lambda_n) = 0 .\]
\end{lemma}

\begin{proof}
  By Proposition~\ref{prop:heightproperty}~(b), the canonical height $\hhat_f(\cdot)$ associated to $f$  is a height function on the projective line $\bP^1$. It follows that
  \begin{equation}
    \hhat_f(c(\lambda)) = (\deg c) \hhat_f(\lambda) +O(1) \quad \text{for all
      $\lambda \in \Kbar$}.
  \end{equation}
  Since by assumption the sequence $(\lambda_n)_{n=1}^\infty$ satisfying
  $f^{\circ n}(\lambda_n) = c(\lambda_n)$ for all $n$, we have   $ (\deg f)^n \hhat_{f}(\lambda_n)= \hhat_f(f^{\circ  n}(\lambda_n)) = \hhat_{f}(c(\lambda_n))$ and thus
\begin{equation}\label{deg f}
(\deg f)^n \hhat_f(\lambda_n) = (\deg c) \hhat_f(\lambda_n) +O(1) \quad \text{for all $n\in \bN$}
\end{equation}
where the implied constant is independent of $n.$

Therefore, $\left((\deg f)^n - \deg c\right) \hhat_f(\lambda_n)$ is bounded by a constant independent of $n. $
Since by assumption $\deg f > 1$, it's clear that $\hhat_f(\lambda_n)$  must go to zero as $n$ goes to infinity.

\end{proof}

\noindent We now state a result about equalities of canonical heights.

\begin{proposition}\label{equal}
Let $K$ be a global field of characteristic zero and let $f, g \in K[x]$ be polynomials of
degrees greater than one.  If there is an infinite nonrepeating
sequence $(\lambda_i)_{i=1}^\infty$, where $\lambda_i \in \Kbar$,
such that
\[ \lim_{i \to \infty} \hhat_f(\lambda_i) = \lim_{i \to \infty}
\hhat_g(\lambda_i) = 0,\]
then $\hhat_f = \hhat_g$.
\end{proposition}
\begin{proof}
  In the case where $K$ is a number field, this is proved in
  \cite[Theorem 3]{Clayton} and \cite[Theorem 1.8]{Mimar}.  The proof
  given in \cite{Clayton} goes through for function fields without any
  changes.  Proofs of similar equalities over function fields appear
  in \cite{GTZ, BD1, GHT1, BD2, YZ}, Thus, we only give a sketch here.
  The idea is to apply equidistribution results such as those in
  \cite{BR, CL, FR1}, all of which hold over both number fields and
  function fields of characteristic 0. For each place $v$ of $K$, the
  $\lambda_i$ equidistribute with respect to the measures of maximal
  entropy $\mu_{f,v}$ and $\mu_{g,v}$ for $f$ and $g$ respectively at
  $v$.  This implies that the local canonical heights $\hhat_{f,v}$ and
  $\hhat_{g,v}$ for $f$ and $g$ are equal to each
  other. By~\eqref{decomposition for the global canonical height}, the
  global canonical heights $\hhat_f$ and $\hhat_g$ are the sum of the
  corresponding local canonical heights. Therefore,  $\hhat_f =
  \hhat_g$, as desired.
\end{proof}

\section{Proof of Theorem \ref{main thm}}\label{nonlin}

 In this section we prove Theorem~\ref{main thm} by first treating the case where $f$ and $g$ both have degrees greater than one.

\begin{prop}\label{nonlinear}
Let $f(x)$ and $g(x)$ be two compositionally independent polynomials with complex coefficients of degree greater than one. Then there are at most finitely many  $\lambda \in \bC$ such that
there are positive integers $m,n$ with the
  following properties:
\begin{enumerate}
\item[(i)]  $f^{\circ m}(x) \not = c(x)$;
\item[(ii)] $g^{\circ n}(x) \not= c(x)$; and
\item[(iii)] $ (x - \lambda) | \gcd(f^{\circ m}(x) - c(x), g^{\circ n}(x) -  c(x)) $.
\end{enumerate}
\end{prop}

\begin{proof}
  Let $K$ be the field generated by all the coefficients of
  $f$, $g$, and $c$ over $\bQ$. Then either $K$ is a number field or a function field of finite transcendence degree over $\Qbar.$ In the latter case, we let $k = K \cap \Qbar$ be its field of  constants.

  We prove the proposition by contraction.
  Suppose that there is an infinite nonrepeating sequence
  $(\lambda_i)_{i=1}^\infty$ such for every $i$, there is an $m_i$ and
  $n_i$ such that $f^{\circ m_i} \not = c$, $g^{\circ n_i} \not= c$, and
  $(x-\lambda_i)$ divides both $f^{\circ m_i}(x) - c(x)$ and
  $g^{\circ n_i}(x) - c(x)$. We will show that the two polynomials $f$ and $g$ must be compositionally dependent. Observe that
  for such $m_i,n_i$, the polynomials $f^{\circ m_i}(x) - c(x)$ and
  $g^{\circ n_i}(x) - c(x)$ have only finitely many roots, so $m_i$ and $n_i$
  must both go to infinity as $i$ goes to infinity.
Then, by Lemma \ref{special}, we have
\[ \lim_{i \to \infty} \hhat_f(\lambda_i) = \lim_{i \to \infty}
\hhat_g(\lambda_i) = 0 .\]
It follows from  Proposition \ref{equal} that  $\hhat_f = \hhat_g$.

Let $\Lambda_0 :=\{\lambda\in \Kbar\mid \hhat_f(\lambda)=0\} = \{\lambda\in \Kbar\mid \hhat_g(\lambda) = 0\}.$  If $K$ is a number field, then by  Proposition~\ref{prop:heightproperty}~(d),
we immediately conclude that $f$ and $g$ share the same set of preperiodic point.
Likewise, if $K$ is a function field and neither $f$ nor $g$ is
isotrivial over $k$, then by \cite{RobFin, BakerFin}, Proposition~\ref{prop:heightproperty}~(d) also holds and hence $f$ and $g$ also share the same set of preperiodic points.

Now assume that at least one of $f$ and $g$ is isotrivial.  Without loss of generality, we assume that $f$ is isotrivial. Since $\hhat_f = \hhat_g$, it follows from the weak Northcott property of \cite{BakerFin} that $g$ is also isotrivial.
Here, we provide an elementary proof of this fact as follows. Since $f$ is isotrivial, there exists a linear polynomial  $\sigma \in \Kbar[x]$ such that $f^{\sigma} = \sigma\circ f \circ \sigma^{-1}\in {\bar k}[x].$ Then, the canonical height $\hhat_{f^{\sigma}}(x)$  associated to $f^{\sigma}$ is equal to the Weil height $h(x)$ of $x\in \Kbar.$ On the other hand,
\[\begin{split}
\hhat_{f^{\sigma}}(\sigma(x)) & = \lim_{n\to \infty}\frac{h\left((f^{\sigma})^{\circ n}(\sigma x) \right)}{d^n} = \lim_{n\to \infty}\frac{h\left(\sigma\circ f^{\circ n}(x) \right)}{d^n} \\ & = \lim_{n\to \infty}\frac{h\left(f^{\circ n}(x) \right)}{d^n} = \hhat_f(x).
\end{split}\]
Thus, $\hhat_f(x) = 0$ if and only if $h(\sigma x) = \hhat_{f^{\sigma}}(\sigma x) = 0.$ In other words, we have $\sigma(\Lambda_0) = {\bar k} = \Qbar.$ Note that $g^{\sigma} : \sigma(\Lambda_0) \to \sigma(\Lambda_0)$ (since $g : \Lambda_0\to \Lambda_0$). We see that $g^{\sigma}(\alpha)\in \Qbar$ for $\alpha \in \Qbar.$ It follows that $g^{\sigma}\in \Qbar[x]$ as well.
Then after conjugating by $\sigma$, we
assume that both $f$ and $g$ are defined over $\Qbar$.
Note that, since each $\lambda_i$ is a solution to
$f^{m_i}(\lambda_i) = g^{n_i}(\lambda_i)$, each $\lambda_i$ must be in
$\Qbar$.  Since $c(\lambda_i)$ is thus in $\Qbar$ for each
$\lambda_i$, and there are infinitely many $\lambda_i$, it follows
that $c \in \Qbar[x]$ as well.

We have reduced to the case where
$K$ is a number field, and we conclude that the set of preperiodic
points of $f$ and $g$ are the same.  This means that the Julia set
$\cJ_f$ and $\cJ_g$ are equal.  By \cite{Baker-Eremenko, SS}, it
follows that unless $f$ and $g$ are both conjugate to a multiple of a
Chebychev polynomial or a multiple of powering map, then there is a
polynomial $q$ and a finite (compositional) order linear map $\tau$ such that any
word in $f$ and $g$ is equal to $\tau^{\circ i} q^{\circ j}$ for
some $i, j$.  This means that $f$ and $g$ must be compositionally
dependent.

Now, we are left with the case where $f$ and $g$ are both conjugate to
either a multiple of a Chebychev polynomial or a multiple of a
powering map.  If $f$ and $g$ are conjugate to $\pm T_{d_1}$ and $\pm
T_{d_2}$, respectively, where $T_{d_i}$ is the monic Chebychev polynomial of
degree $d_i$, then $f$ and $g$ are compositionally dependent (easy to
check).  If $f$ and $g$ are both conjugate to powering maps, then
after conjugation we may write $f(x) = x^{d_1}$ and $g(x) = \gamma
x^{d_2}$ for some $\gamma \in \Qbar$. Note that both $f$ and $g$ have the
same set of preperiodic points which are all the roots of unity in this case.
In particular, $\gamma = g(1)$ is a root of unity. Therefore $f$ and $g$ must be compositionally dependent as well.

\end{proof}

\noindent
Next, we treat the case where exactly one of $f$ and $g$ is linear.

\begin{prop}\label{onelinear}   Let $f(x)$ and $g(x)$ be two polynomials
  of $\C[x]$ such that $\deg f > 1$ and $\deg g = 1$.  Then there are at most finitely many
  $\lambda \in \bC$ such there are positive integers $m,n$ with the
  following properties:
\begin{enumerate}
\item[(i)]  $f^{\circ m} \not = c(x)$;
\item[(ii)] $g^{\circ n} \not= c(x)$; and
\item[(iii)] $ (x - \lambda) | \gcd(f^{\circ m}(x) - c(x), g^{\circ n}(x) -
c(x)) $.
\end{enumerate}
\end{prop}
\begin{proof}
  Let $K$ be the field generated by the coefficients of $f$, $g$, and
  $c$.  Since $g^{\circ n}(x) - c(x)$ is a polynomial of degree at
  most $\deg c + 1$, we see that every $\lambda$ such that
  $g^{\circ n}(\lambda) - c(\lambda) = 0$ has degree at most
  $\deg c +1$ over $K$. Note that,  for any nonrepeating infinite sequences
  $(\lambda_i)_{i=1}^\infty$ and $(n_i)_{i=1}^\infty$ such that
  $f^{\circ n_i}(\lambda_i) = c(\lambda_i)$ for all $i$, we have
  $\lim_{i \to \infty} \hhat_f(\lambda_i) = 0$ by  Lemma~\ref{special}. If $K$ is a number field, then by Northcott property we conclude that there are only at most finitely many $\lambda$ that satisfy properties~(i) to (iii) given above. Hence,  the proposition holds in this case.

  Now, let's assume that $K$ is a function field and that there is a  nonrepeating infinite sequences $(\lambda_i)_{i=1}^\infty$ and $(n_i)_{i=1}^\infty$ such that
  $f^{\circ n_i}(\lambda_i) = c(\lambda_i)$ for all $i\in \bN.$
  We note that as in the proof of Proposition~\ref{nonlinear}, both $m_i$ and
  $n_i$ must go to  infinity since $c(x)$ is not a compositional power of
  $f$ or $g$.

  By \cite{BakerFin}, if there is an infinite sequence of
  $(\lambda_i)_{i=1}^\infty$ of bounded degree with $\hhat_f(\lambda_i) = 0$ then $f$ must be isotrivial.
  Thus, after changing variables, we may assume that $f \in k[x]$ for
  some number field $k$. As a consequence,  $\hhat_f(x)= h(x)$ the
  Weil height of $x$ for all $x\in \Kbar.$ On the other hand, it
  follows from the definition of Weil height that for $x\in \Kbar$
  with $h(x) > 0$ we must have $h(x) \ge 1/(\deg x).$ Now the sequence
  $(\lambda_i)_{i=1}^\infty$ has the property that all $\lambda_i$
  have degrees bounded above by $\deg c + 1$ over $K$ and that
  $\lim_{i\to \infty}\, h(\lambda_i) = 0$. Therefore we must have
  $h(\lambda_i) = 0$ for all but finitely many $i.$ Also note that for
  $x\in \Kbar$ we have $h(x) = 0$ if and only if $x\in {\bar k} = \Qbar.$ So, for all but finitely many $\lambda_i$ in the sequence $(\lambda_i)_{i=1}^\infty$ must be in $\Qbar.$

We  are left to treat the case where there are infinitely many
$\lambda$   in $\Qbar$ such that
  $f^{\circ m}(\lambda) = c(\lambda) = g^{\circ n}(\lambda)$.  We see
  in this case that $c$ must have coefficients in $\Qbar$ since there
  are infinitely many $\lambda \in \Qbar$ such that
  $c(\lambda) \in \Qbar$.  Let $k$ be the field generated by the
  coefficients of $f$ and $c$ over $\bQ$, and let $g(x) = \alpha x +
  \beta$.   Then all $\lambda$ such that
  $f^{\circ m}(\lambda) = c(\lambda) = g^{\circ n}(\lambda)$ lie in
  extensions of $\Qbar \cap k(\alpha, \beta)$ having degree at most
  $\deg c + 1$ .  Since $\Qbar \cap k(\alpha,\beta)$ is a finitely generated
  extension of $k$, all such $\lambda$ have bounded degree over $\bQ$.
  Since the $\lambda$ also have bounded height, again we have a
  contradiction by Northcott's theorem.
\end{proof}

\begin{proof}[Proof of Theorem \ref{main thm}]

If $\deg f, \deg g > 1$, then Theorem \ref{main thm} follows
immediately from Proposition \ref{nonlinear}.  If $\max(\deg f, \deg g) > 1$ and
$\min(\deg f, \deg g) = 1$, then we may assume without loss of
generality that $\deg f > 1$ and $\deg g =1$.  Theorem \ref{main thm}
then follows from Proposition \ref{onelinear}.

\end{proof}

\section{Proof of Theorem \ref{linear-thm}} \label{linear}

When $f$ and $g$ are both linear, there may be infinitely many
$\lambda$ such that $(x - \lambda)$ divides
$\gcd(f^{\circ m}(x) - c(x), g^{\circ n}(x) - c(x))$ for some $m$ and
$n$.  Take for example, $c(x) = x^2$, with $f(x) =2x$ and $g(x) =
x+1$.  Then
\[ f^{\circ n}(x) - c(x) = 2^n x - x^2 = -x(x-2^n) \]
while if $m = 2^n(2^n - 1)$, then
\[ g^{\circ m}(x) - c(x) = x +  2^n(2^n  - 1) - x^2 = -
(x+2^n-1)(x-2^n), \]
so clearly there are infinitely many $\lambda$ such that $(x - \lambda)$ divides
$\gcd(f^{\circ m}(x) - c(x), g^{\circ n}(x) - c(x))$ for some positive
integers $m$ and $n$.  On the other hand, if we restrict to the case
where $m=n$, then we may obtain a suitable finiteness result.

The techniques in this section are mostly from diophantine geometry.
We use these to prove Proposition \ref{algebraic} which treats the
case where the coefficients of $f$, $g$, and $c$ are algebraic.  We
then derive Theorem \ref{linear-thm} using some simple specialization
arguments.  Theorem \ref{just-n} then follows from Theorem
\ref{linear-thm} and Propositions \ref{nonlinear} and
\ref{onelinear}.

\subsection{Results from diophantine geometry}

We will use the following version of Roth's Theorem (see
\cite[Chap.~7~Thm.~1.1]{dio-geo} and Remark (v) following it).

\begin{theorem}\label{thm:roth}
Let $k$ be a number field, let $\alpha_1, \dots, \alpha_n$ be
distinct points in $k$, and let $S$ be a finte set of places of $k$.
Then for any $\epsilon > 0$, there are at most finitely many $\beta
\in k$
such that
\begin{equation}
\begin{split}
 \frac{1}{[k:\bQ]} & \left( \sum_{v \in S} \sum_{i=1}^n  - \min (\log \| \alpha_ i -
  \beta \|_v, 0)  +  \sum_{v \in S} \max ( \log \| \beta
  \|_v, 0) \right) \\
& \geq (2+\epsilon) h(\beta)
\end{split}
\end{equation}
\end{theorem}

The following is Siegel's well-known theorem on the set of integral
points of curves of genus zero, which can be derived from Theorem
\ref{thm:roth} without difficulty.   We refer the reader
to~\cite[Chap. 8 Theorem~5.1]{dio-geo} for a proof.

\begin{theorem}\label{thm:integral points}
  Let $k$ be a number field. Let $C$ be a complete
  non-singular curve of genus 0, defined over $k$, let $S$ be a finite
  set of places of $k   $ containing all the archimedean places, and let $\phi$ be a
  non-constant function in $k(C)$ with at least three distinct poles.    Then
  there are at most finitely many $Q \in C(k)$ such that $\phi(Q)$ is
  an $S$-integer.
\end{theorem}

As a corollary to Theorem~\ref{thm:integral points}, we have the
following, which we will use to treat the case where the coefficients of the linear
terms of $f$ and $g$ are multiplicatively dependent.

\begin{proposition}\label{prop:s-integer}
  Let $W$ be a one dimensional subtorus in $\bG_m^2$ defined over a
  number field $k$ and let $S$ be a finite set of places of $k$
  containing all the archimedean places. Let
  $\Phi(X,Y)= P(X,Y)/Q(X,Y)$ where $P, Q\in k[X,Y]$ are two relatively
  prime polynomials neither of which is divisible by $X$ or
  $Y$. Assume that $\Phi$ restricts to a non-constant rational
  function $\phi$ on $W$ with at least a pole in $W(\kbar)$. Let
  $\Gamma$ be a finitely generated subgroup of $W(k).$ Then, there are
  at most finitely many points $Q\in \Gamma$ such that $\phi(Q)$ is an
  $S$-integer.
\end{proposition}

\begin{proof}
  Here, as usual, we consider $\bG_m^2$ to be the open subset of
  $\bP^2$ with coordinates $[x:y:z]$ defined by $x \not= 0$,
  $y \not= 0$, $z\not=0$.  The functions $X$ and $Y$ are equal to
  $x/z$ and $y/z$ with respect to these coordinates.
  By making a finite extension of $k$, we  assume that the poles of $\phi$ are all
  $k$-rational points of $W.$ Moreover, because $\Gamma$ is finitely generated, we
  may assume, possibly after extending $S$ to a larger finite set of
  places, that all of the elements of $\Gamma$ as well as the poles of $\phi$ whose coordinates are $S$-units.
  Possibly by enlarging $S$, we may also assume that that
  the poles of $\phi$ whose coordinates are also $S$-units. Let $\Gamma^\ast$ be the union
  of $\Gamma$ and the set of poles of $\phi.$

  Now, we fix a positive integer $m \ge 2$ and let $\mu_m : \bG_m^2 \to \bG_m^2$ be the
  $m$-th powering map. Namely, $\mu_m(X,Y) = (X^m, Y^m)$ for all $(X,Y)\in \bG_m^2.$  By Kummer theory,
  there exists a finite extension $L$ over $k$ such that the inverse image $\mu_m^{-1}\left(\Gamma^\ast\right)$
  of $\Gamma^\ast$ is contained in $W(L).$ Let $S'$ denote the set of places of $L$ that extend the places in $S.$

  As $\mu_m : W \to W$ is an unramified map of degree $m^2,$ we see
  that the the function $\phi_m := \phi\circ \mu_m$  is a rational function with at least $m^2 $ distinct poles on $W.$
  The subtorus  $W$ is viewed as an affine curve in the projective plane $\bP^2_k$ and we denote
  its Zariski closure in $\bP^2$ by $\overline{W}$. Note that $\phi_m$ extends to a rational function on $\overline{W}$
  which we still denote by $\phi_m.$
  Let $\pi : \widetilde{W} \to \overline{W}$ denote the
  normalization of $\overline{W}.$ Then, ${\widetilde W}$ is a
  projective smooth curve of genus 0. Furthermore, the function $\psi_m := \phi_m \circ \pi$ is a rational function
  on ${\widetilde W}$ with at least $m^2$ distinct poles. On the other hand, the set of $L$-rational points $W(L)$ lift
  to the set ${\widetilde W}(L).$

  Observe that for any point $Q \in \Gamma$ such that $\phi(Q)$ is an $S$-integer, then $\psi_m(Q')$ is an $S'$-integer
  where $Q'\in {\widetilde W}(L)$ is any point such that $\left(\mu_m\circ\pi\right)(Q') = Q$. On the other hand,
  since $m^2  > 3$, there are at most finitely many $Q' \in \widetilde{W}(L)$ such that $\psi_m(Q')$ is an $S'$-integer
  by Theorem~\ref{thm:integral points}. Thus, there are at most
  finitely many $Q$ such that $\phi(Q)$ is an $S$-integer.

\end{proof}

We will use the following Lemma, due originally to Siegel \cite{Siegel29}.  We provide a proof in modern language for the sake of completeness.

\begin{lemma}\label{Siegel}
Let $w$ be element of a number field $k$, let $y$ be a nonzero element
of $k$, and let $S$ be a finite set
of places of $k$ including all the archimedean places.  Let $\epsilon
> 0$.  Then
\begin{equation}\label{SiegelEq}
\frac{1}{[k:\bQ]} \sum_{v \notin S} - \min(\log \|w^n - y \|_v, 0) \geq (1-\epsilon) n h(w)
\end{equation}
for all sufficiently large $n$.
\end{lemma}

\begin{proof}
We may assume that $S$ contains all the places $v$ of $k$ such that
$\|w\|_v \not= 1$.   Then applying Theorem \ref{thm:roth}, to the points $0$
and $y$, we see that for any $\epsilon > 0$, we have

\begin{equation*}
\begin{split}
 \frac{1}{[k:\bQ]} \sum_{v \in S}  & \left (- \min(\log \|w^n - y\|_v, 0) - \min(\log
  \|w^n\|_v,0)  +  \max (\log \|w^n\|_v,0) \right)  \\  & \leq (2+ \epsilon)n  h(w) + O(1).
\end{split}
\end{equation*}
Since $S$ contains all places such that $\|w\| \not =1$, we have
\[ \frac{1}{[k:\bQ]} \sum_{v \in S} \left( - \min(\log \|w^n\|_v,0) +
\max (\log \|w^n\|_v,0) \right) = 2 n h(w). \]
Thus,
\begin{equation}\label{eq:roth}
 \frac{1}{[k:\bQ]} \sum_{v \in S} - \min(\log \|w^n - y\|_v, 0) \leq \epsilon n h(w) +
O(1).
\end{equation}
Since
\[ n h(w) \leq \frac{1}{[k:\bQ]} \sum_{v \in \Omega_k} - \min(\log \|w^n - y\|_v, 0) +O(1),\]
we see that \eqref{SiegelEq} must hold.
\end{proof}

The following lemma will be used to treat the case where the  coefficients of the linear terms of $f$ and $g$ are multiplicatively independent.

\begin{lemma}\label{from-BCZ}
Let $w_1$ and $w_2$ be two multiplicatively independent elements
of a number field $k$, neither of which is a root of unity, and let
$y$ be a nonzero element of $k$.  Let $S$ be a finite set of places of
$k$ including all the archimedean places.  Then for all sufficiently
large $n$, there is a $v \notin S$ such that $|w_1^n - y|_v < |w_2^n - y|_v \le 1$.
\end{lemma}

\begin{proof}
We begin by showing that if $w_1$ and $w_2$ are multiplicatively
independent, then $w^n_1/y$ and $w^n_2/y$ are multiplicatively independent
for all but at most finitely many $n$.  Note that if $y$ is not in the
multiplicative group generated by $w_1$ and $w_2$, then  $w^n_1/y$
and $w^n_2/y$ are multiplicatively independent for all $n$.
Otherwise, we have $y^{\ell_1} = w_1^{\ell _2} w_2^{\ell _3}$ for some
integer $\ell_1 > 0$ and some integers $\ell_2$ and $\ell_3$.  Since
it suffices to prove our lemma for $\ell_1$-th roots of $w_1$ and
$w_2$ we may assume that we have $y = w_1^i w_2^j$ for some integers
$i$, $j$.  Now, if $w_1^n/(w_1^i
w_n^j)$ and $w^n_2/(w_1^i w_2^j)$ are multiplicatively dependent,
then we must have $(n-i)(n-j) = (-i)(-j)$, since $w_1$ and $w_2$ are
multiplicatively independent.  For all sufficiently large $n$, we
clearly have $(n-i)(n-j) > (-i)(-j)$, so we are done.

By Theorem 1 and equation (1.2) of \cite{CZ05}, we see that for any $\epsilon
> 0$, there is a constant $C_\epsilon$ such that
\begin{equation}\label{eq:gcd}
 \frac{1}{[k:\bQ]} \sum_{v \in \Omega_k} - \log^- \max(\|w_1^n - y\|_v, \|w_2^n - y\|_v ) <
 \epsilon n h(w_1) + C_\epsilon,
 \end{equation}
where $\log^-(\cdot) = \min( 0, \log (\cdot))$.
We may enlarge $S$ to include the place $v$  where $|w_1|_v> 1$ or
$|y|_v> 1$. Suppose that for a positive integer $n$, inequalities $|w_1^n - y|_v \ge |w_2^n - y|_v$ hold for all $v\not\in S$. Then, from \eqref{eq:gcd} we have that
\begin{align*}
  \epsilon n h(w_1) + C_\epsilon  & \ge \frac{1}{[k:\bQ]} \sum_{v \in \Omega_k} (- \min\left( 0,
      \max\{\log\|w_1^n - y\|_v,\log\|w_2^n - y\|_v\} \right)
  \\ & \ge \frac{1}{[k:\bQ]} \sum_{v \not\in S} -\min\left(0, \log\|w_1^n - y\|_v\right) \ge (1-\epsilon) n h(w_1),
\end{align*}
where the last inequality follows from \eqref{SiegelEq}.
Taking $\epsilon = 1/3$, we see that there are only finitely many positive integers $n$ such that the above inequality holds. Hence,  for all sufficiently large $n$ there is a $v \not\in S$ such that $|w_1^n - y|_v < |w_2^n - y|_v \le 1$, as desired.
\end{proof}

\subsection{Proofs of Theorem \ref{linear-thm} and
  \ref{just-n}}

We are now ready to treat the case where $f$, $g$ are linear polynomials, and  $f, g$ and $c$ all have  algebraic coefficients.  The proof breaks into several cases.  The
first case is when $c$ is constant; this case is already treated in
\cite{Mike1}.  The idea in all of the other cases is the same: to force
certain quantities coming from any solutions to
$f^{\circ n}(x) = c(x) = g^{\circ n}(x)$ to have poles outside a
finite set and then derive contradictions from the existence of these
poles to show that there are no solutions to
$f^{\circ n}(x) = c(x) = g^{\circ n}(x)$ when $n$ is sufficiently
large.

\begin{proposition}
\label{algebraic}
  Let $f(x) = \alpha x$ and $g(x) = \beta x + \gamma$ where $\alpha$,
  $\beta$, and $\gamma$ are nonzero algebraic numbers such that
  $\alpha$ is not a root of unity, $\alpha \beta$ is not a root of
  unity, $\beta$ is not a root of unity other than 1,
  and $\gamma \not= 0$.  Let $c(x)$ be any polynomial with
  coefficients in $\Qbar$.  Then for all but at most finitely many
  $n$, we have
\begin{equation} \label{1}
\gcd(f^{\circ n}(x) - c(x), g^{\circ n}(x) - c(x)) =1
\end{equation}
\end{proposition}

\begin{proof}
Suppose that there are infinitely many $n$ such that \eqref{1} does not hold. Let $n$ be an integer such that $\gcd(f^{\circ n}(x) - c(x), g^{\circ n}(x) - c(x)) \ne 1$. Then there exists  a $\lambda_n\in \Qbar$  such that
$$(x-\lambda_n)\mid \gcd(f^{\circ n}(x) - c(x), g^{\circ n}(x) - c(x))$$
and thus, $f^{\circ n}(\lambda_n) = c(\lambda_n) = g^{\circ n}(\lambda_n).$

In the following, we break the proof into four cases and show a contradiction in each case.
\smallskip
\\
\noindent {\bf Case I.}  Suppose that $c$ is a constant.
Let $\theta$
be the compositional inverse of $f$ and let $\tau$ be the
compositional inverse of $g$.  We observe that
if 
$f^{\circ n}(\lambda) = g^{\circ n}(\lambda) = c$ then $\theta^{\circ n}(c) = \tau^{\circ
  n}(c)$.  By \cite[Proposition 5.4]{Mike1}, this implies that either
$\theta$ and $\tau$ have a common iterate or that $c$ is periodic
under both $\theta$ and $\tau$.  Since $\theta = \alpha^{-1}x$, we see that
zero is the only periodic point of $\theta$.  Since $\tau = x/\beta -
\gamma/\beta$, we see
that the constant term of $\tau^{\circ n}$ is always nonzero, so 0
cannot be a periodic point of $\tau$.  Thus, there is an $n$ such that
$\theta^{\circ n} = \tau^{\circ n}$, which means that $f$ and $g$ have a common iterate.
Since the constant term of $g^{\circ n}$ is nonzero for all $n$, we
see that $f$ and $g$ cannot have a common iterate, which gives a
contradiction.
\smallskip

In the following,  we assume that $\deg c \geq 1.$
\smallskip
\\
\noindent{\bf Case II.}  Assume that  $\beta=1$.
Then
\[ \lambda_n = \frac{n \gamma}{\alpha^n - 1}. \]
Let $S$ be the set of places $v$ that are archimedean or where
$\alpha$, $\gamma$, or a coefficient of $c$ has $v$-adic absolute
value not equal to 1.  Assume that $\lambda_n$ is an $S$-integer.
Then,
\begin{align}
h(\lambda_n) & = \frac{1}{[K: \bQ]} \sum_{v\in S}\; \max\left(0, \log \left\|\frac{n \gamma}{\alpha^n - 1}\right\|_v\right) \notag\\
& \le \frac{1}{[K: \bQ]} \sum_{v\in S}\;\left\{\max\left(0, \log \left\|\frac{1}{\alpha^n - 1}\right\|_v\right)+\max\left(0,\log\|n \gamma\|_v \right)  \right\} \notag\\
& = \frac{1}{[K: \bQ]} \sum_{v\in S}\;  \max\left(0, \log \left\|\frac{1}{\alpha^n - 1}\right\|_v\right)+ h(n\gamma)  \notag \\
& \le \frac{1}{[K: \bQ]} \sum_{v\in S}\;  \max\left(0, \log \left\|\frac{1}{\alpha^n - 1}\right\|_v\right) + \log n + O(1) \label{eq:logn}
\end{align}
Let $\epsilon >0$ be given. By \eqref{eq:roth}, there exists a constant $C_{\epsilon}$ such that
\begin{equation}\label{eq:siegel2}
\frac{1}{[K: \bQ]} \sum_{v\in S}\;  \max\left(0, \log \left\|\frac{1}{\alpha^n - 1}\right\|_v\right) \le \epsilon n h(\alpha) + C_\epsilon.
\end{equation}
On the other hand, there is a constant $D= D(\gamma)$ such that
$$h(\lambda_n) = h( n\gamma/(\alpha^n - 1)) \ge n h(\alpha) - h(n) - D = n h(\alpha) - \log n - D.$$

Fixing a positive $\epsilon < 1$ and combing  \eqref{eq:logn} with \eqref{eq:siegel2},
we see that $\lambda_n$ can not be an $S$-integer if $n$ is large enough, . Therefore, for $n$ large there exists a place $v$ out side of $S$ such that $|\lambda_n|_v > 1.$
If $\deg c > 1$, then
$|c(\lambda)|_v = | \lambda_v |^{\deg c}$ but
$|f^{\circ n}(\lambda_n)| = |\alpha^n \lambda_n|_v = |\lambda_n|_v$. This gives a contradiction.

If $\deg c = 1$, then  we write $c(x) = tx + u$ and note that since
$f^{\circ n}(\lambda_n) = g^{\circ n}(\lambda_n) = c(\lambda_n)$, we
must have
\[\lambda_n = \frac{u - n \gamma}{1 -t} = \frac{u}{\alpha^n - t}\]
If $u \not = 0$ and $n$ is large, then by enlarging $S$ to contain the places $v$ where $|1-t|_v \ne 1$ , then $(u - n \gamma)/(1-t)$ is an $S$-integer for all $n$. On the other hand, by taking $\Phi(X,Y) = u/(X - t)$ in Proposition~\ref{prop:s-integer}, we see that $u/(\alpha^n -t)$ can not be an $S$-integer for $n$ sufficiently large. This gives a contradiction.
 If $u = 0$, then we have $\lambda_n
= \alpha^n \lambda_n = t \lambda_n = g^{\circ n}(\lambda_n)$, which has
no solutions when $\alpha^n \not= t$, and thus has a solution for at
most one $n$, since $\alpha$ is not a root of unity. Thus the proof of this case is completed.
\smallskip

We assume in the following that $\beta \ne 1$. Note that when $\alpha^n = \beta^n$, there is no solution to $f^{\circ
  n}(x) = g^{\circ n}(x)$ and that when $\alpha^n \not= \beta^n$,  the
unique solution to $f^{\circ n}(x) = c(x) = g^{\circ n}(x)$ is given by
\begin{equation}\label{n}
\lambda_n = \frac{(\beta^n -1)\gamma}{(\beta-1)(\alpha^n - \beta^n)}.
\end{equation}
\smallskip
\\
\noindent{\bf Case III.}  Suppose that $\alpha$ and $\beta$ are multiplicatively dependent.
Then, the point $P = (\alpha, \beta)$ is in a one dimensional subtorus $W$ of $\bG_m^2$.
Let $S$ be the set of places $v$ that are archimedean or where $\alpha$,
$\gamma$, $\beta-1$,  or a coefficient of $c$ has $v$-adic absolute value not
equal to 1.   Then, by taking $\Phi(X,Y) = (Y-1)/(X-Y)$ and $\Gamma$ to be the group generated by $P$ in Proposition~\ref{prop:s-integer}, we see that for all sufficient large $n$ there exists a place $v$ outside of $S$ such  that $$\left|(\beta^n -1)/(\alpha^n - \beta^n)\right|_v > 1.$$
It follows that  for such $v$ we have $|\lambda_n|_v > 1.$
Observe that on the one hand, $\left|f^{\circ n}(\lambda_n)\right|_v = |\alpha^n \lambda_n|_v = |\lambda_n|_v$ while on the other hand, we have $\left|f^{\circ n}(\lambda_n)\right|_v = |c(\lambda_n)|_v = |\lambda_n|_v^{\deg c}.$ This gives a contradiction if $\deg c > 1.$

If $\deg c = 1$, we write  $c(x) = tx + u, \,t\ne 0$. If $f^{\circ n}(\lambda_n) = c(\lambda_n) = g^{\circ n}(\lambda_n)$  then we have
\begin{equation}\label{ii}
\lambda_n = \frac{u - (\beta^n-1) \gamma/(\beta - 1)}{\beta^n -t} = \frac{u}{\alpha^n - t}
\end{equation}
From this we deduce that
\begin{equation}\label{eq:ii}
\frac{\beta^n - t}{\alpha^n - t} = u - \left(\frac{\gamma}{\beta -1}\right)(\beta^n -1).
\end{equation}
Note that the right hand side of \eqref{eq:ii} is an $S$-integer. However, by taking $\Phi(X,Y) = (Y-t)/(X-t)$ in  Proposition~\ref{prop:s-integer} we conclude that for $n$ large enough the left hand side of \eqref{eq:ii} is not an $S$-integer.  This leads to a contradiction and completes the proof in this case.
\smallskip
\\
\noindent {\bf Case IV.}  Suppose that  $\alpha$ and $\beta$ are multiplicatively independent.  Let $S$ be the
set of places $v$ that are archimedean or where $\alpha$, $\gamma$, or
a coefficient of $c$ has $v$-adic absolute value not equal to 1.

Suppose that $\deg c > 1$.  Then, applying Lemmas \ref{from-BCZ} to
$\beta^n -1$ and $(\alpha/\beta)^n - 1$, we see that there is a place
$v$ outside of $S$ such that $|\lambda_n|_v > 1$.  Again, if
$\deg c > 1$, this gives a contradiction since we have
$|c(\lambda)|_v = | \lambda_v |^{\deg c}$ but
$|f^{\circ n}(\lambda_n)| = |\alpha^n \lambda_n|_v = |\lambda_n|_v$.

Now suppose that $\deg c = 1$.  Again, we write $c(x) = tx + u$. Then we also have
\begin{equation}\label{last}
\lambda_n = \frac{u - \gamma(\beta^n-1)/(\beta-1)}{\beta^n -t} =
\frac{u}{\alpha^n - t}.
\end{equation}
This is equivalent to
\begin{equation}\label{lastlast}
1- \frac{\gamma(\beta^n-1)}{u (\beta-1)} = \frac{\beta^n - t }{\alpha^n - t}.
\end{equation}
We enlarge $S$ to include all the places such that $u$ or $\beta - 1$
are $S$-unit.  Then applying Lemma \ref{from-BCZ}, we see that for all
sufficiently large $n$, there is a place $v \not \in S$ such that
$|\alpha^n - t|_v < |\beta^n - t|_v \leq 1$. For this $v$, we see that
the left hand side of \eqref{lastlast} is a $v$-adic integer while the
right hand side is not. Therefore, \eqref{last} can not hold for $n$
sufficiently large.
\end{proof}

\begin{remark}
To see that
Proposition~\ref{algebraic} does not hold in general if  $\alpha \beta$ is
a root of unity, consider the case where $f(x)
= x/2$, $g(x) = 2x + 1$ and $c(x) = - (x+1)$.  Then
for any $n$, the common root of $f^{\circ n}$ and $g^{\circ n}$ is
\[ \frac{2^n - 1}{2^{-n} - 2^n} = - 2^n \frac{2^n - 1}{2^{2n} - 1} =
\frac{-2^n}{2^n + 1}. \]
while the common root of $f^{\circ n}$ and $c(x)$ is
\[ \frac{-1}{(1/2)^n + 1} =  \frac{-2^n}{2^n + 1}.\]
Thus, for every positive integer $n$, there is a $\lambda_n$ such that
\[ f^{\circ n}(\lambda_n) - c(\lambda_n) = g^{\circ n}(\lambda_n) -
c(\lambda_n) = 0.\]
\end{remark}

We can now prove Theorem \ref{linear-thm} by specializing from $\bC$
to a number field.

\begin{proof}[Proof of Theorem \ref{linear-thm}]

  First we note that any nonconstant affine map $x \mapsto ax + b$ has a fixed
  point unless $a = 1$.  Any two monic linear polynomial
  must commute with each other.
  Thus, we may assume that at least one of $f$ and $g$ has a fixed
  point.  Without loss of generality, we may assume that $f$ has a
  fixed point.  After a possible change of coordinates, we may then write
  $f(x) = \alpha x$ and $g(x) = \beta x + \gamma$.

  If $\alpha$ is a root of unity, then $f$ and $g$ are not
  compositionally independent since $f$ itself is compositionally
  torsion, so $\alpha$ must not be a root of unity.  Similarly, if
  $\beta$ is a root of unity other than one, then $g$ is
  compositionally torsion so that $f$ and $g$ are not compositionally
  independent either.  We may therefore assume that $\beta$ is not a
  root of unity other than one.  Finally, we see that if there are
  integers $i$ and $j$ such that $\alpha^i \beta^j = 1$, then the
  linear terms in $f^{\circ i} g^{\circ j}$ and
  $g^{\circ j} f^{\circ i}$ are both 1, which means that
  $f^{\circ i} g^{\circ j}$ and $g^{\circ j} f^{\circ i}$ commute.
  This would imply $f$ and $g$ are not compositionally dependent, so
  we may assume that there are no positive integer $i$ and $j$ such
  that $\alpha^i \beta^j = 1$.

  As in the proof of Proposition~\ref{algebraic}, we assume that there are infinitely many $n$
  such that \eqref{1} does not hold.
  Let $K$ be the field generated by $\alpha, \beta, \gamma$ over
  $\Q$, and let $R$ be the ring generated over $\bZ$ by $\alpha,
  \beta, \gamma$ and the coefficients of $c$.  Observe
  that any solution $\lambda_n$ to $f^{\circ n}(\lambda_n) = g^{\circ
    n}(\lambda_n) = c(\lambda_n)$ must lie in $K$. By our assumption, there are
  infinitely many such $n$, so $c$ takes infinitely many values in
  $K$ to other values in $K$ so $c \in K[x]$.  Hence, we may assume
  that $c \in K[x]$.

  If $\alpha$, $\beta$, and $\gamma$ are in $\Qbar$, then we are done by
  Proposition \ref{algebraic}. If $K$ has positive transcendence degree
  over $\Qbar$, then there exists a specialization map $t$ from $R$ to
  $\Qbar$ such that $\gamma_t \ne 0$ and $\alpha_t$, $\beta_t$, $\alpha_t\beta_t$, and
  $\alpha_t/\beta_t$ are not roots of unity.  We may prove this, for
  example, by induction on the transcendence degree of
  $\Q(\alpha, \beta, \gamma)$.  If the transcendence degree is 0,
  there is nothing to prove.  If it is $n$, take a subfield $L$ of
  transcendence degree of $n-1$ in $K$.  Then, by \cite[Theorem
  4.1]{CS}, for all specializations $s$ from $R$ to ${\bar L}$ of
  sufficiently large height, we have that $\gamma_s \ne 0$ and that
  $\alpha_s$, $\beta_s$, $\alpha_s\beta_s$, and  $\alpha_s/\beta_s$ are not roots of
  unity.  We then the inductive hypothesis on the transcendence
  degree to $\Q(\alpha_s, \beta_s, \gamma_s)$.

Let $f_t = \alpha_t x$, $g_t = \beta_t x + \gamma_t$, and $c_t$ be the
polynomial obtained by specializing all the coefficient of $c$
at $t$.  Now,  if $\gcd(f^{\circ n}(x) - c(x), g^{\circ n}(x) -
c(x)) \not= 1$, then $\gcd(f_t^{\circ n}(x) - c_t(x), g_t^{\circ n}(x)
- c_t(x)) \not = 1$.  But there are at most finitely many $n$ such
that $\gcd(f_t^{\circ n}(x) - c_t(x), g_t^{\circ n}(x)
- c_t(x)) \not = 1$, by Proposition \ref{algebraic}, which gives a
contradiction, and finishes our proof.

\end{proof}

\begin{remark}
  We note that by Proposition \ref{algebraic}, the condition needed
  for Theorem \ref{linear-thm} is weaker than merely compositional
  dependency, since Proposition \ref{algebraic} holds unless the
  linear term of $f \circ g$ is a root of unity.  Mike Zieve has shown
  us that something similar is true for polynomials of higher degree,
  namely that the sorts of compositional dependencies that may arise
  all take a specific form.
\end{remark}

\noindent We now prove Theorem \ref{just-n}.
\begin{proof}[Proof of Theorem \ref{just-n}]
  The case where $f$ and $g$ are both linear is covered by Theorem
  \ref{linear-thm}, so we may assume that either both $f$ and $g$ are
  nonlinear or that $g$ is linear and $f$ is not.

  By Propositions \ref{nonlinear} and \ref{onelinear}, there are at
  most finitely many $\lambda$ such that $(x-\lambda)$ divides
  $\gcd(f^{\circ n}(x) - c(x), g^{\circ n}(x) - c(x))$ for some $n$
  such that $f^{\circ n} \not= c$ and $g^{\circ n} \not= c$.  Let
  $\cS$ denote the set of such $\lambda$.  Since $f$ and $g$ are
  compositionally independent, there is at most one $N$ such that
  $f^{\circ N} = c$ or $g^{\circ N} = c$ exclusively.  If such an $N$ exists,
  let $\cT$ denote the set of $\lambda$ such that $(x-\lambda)$
  divides $\gcd(f^{\circ N}(x) - c(x), g^{\circ N}(x) - c(x))$. We
  observe that $\cT$ must be finite since otherwise we would have
  $f^{\circ N} - c=0 =  g^{\circ N} - c$. However, this cannot happen
  because $f$ and $g$ are compositionally independent.
  Any $\lambda$ such that $(x- \lambda)$divides
  $\gcd(f^{\circ n}(x) - c(x), g^{\circ n}(x) - c(x))$ is in
  $\cS \cup \cT$, so our proof is complete
\end{proof}

\section{Proof of Theorem \ref{gcdthm}} \label{gcd}

Theorem \ref{gcdthm} is now an easy consequence of the following lemma.
To state the lemma, we introduce a small bit of new notation: for any
nonzero polynomial $q(x)$ we let $v_\lambda(q)$ denote the largest
positive integer $e$ such that $(x-\lambda)^e$ divides $q$ when $(x-
\lambda) |q$ and let $v_\lambda(q) = 0$ if $(x-\lambda)$ does not
divide $q$.

\begin{lemma}\label{close}
  Let $q$ be a polynomial in $\C[x]$ of degree greater than one and
  let $c(x) \in \bC[x]$ be a polynomial that is not equal to a
  constant that is in a ramified cycle of $f$.  Let $\lambda \in \bC$.
  Then there is a constant $M_{\lambda,q}$ such that
  $v_\lambda(q^{\circ n}(x) - c(x)) \leq M_{\lambda, q}$ for all $n$
  such that $q^{\circ n}(x) \not= c(x)$.
\end{lemma}
\begin{proof}
  We write $c(x) = \sum_{i=0}^{d_c} c_i (x-\lambda)^i$ as a polynomial in $(x-\lambda)$.
  If there are finitely many
  $n$ such that $v_\lambda(q^{\circ n}(x) - c(x)) > 0$, then the proof
  is immediate.  Thus, we assume that there are infinitely many $n$
  such that $v_\lambda(q^{\circ n}(x) - c(x)) > 0$.  It follows that
   $q^{\circ n}(\lambda) = c_0$ for infinitely many $n$, so $c_0$
   must be periodic under $q$.  Let $\ell$ be the smallest positive
  integer such that $q^{\circ \ell}(\lambda) = c_0$ and let $r$ be the
  smallest positive integer such that $q^{\circ r}(c_0) = c_0$.  Then
  we see that $v_\lambda(q^{\circ n} (x) - c(x)) > 0$ if and only if
  $n$ can be written as $\ell + kr$ for some $k$.  We write
  $q^{\circ r}(x) = \sum_{i=0}^{d_r} a_i (x-c_0)^i$ and
  $q^{\circ \ell}(x) = \sum_{j=0}^{d_\ell} b_j (x- \lambda)^j$. Let $e$ be
  the smallest positive integer such that $b_e \not= 0$.

  Suppose now that $c(x) = c_0$ is a constant. By assumption, $c_0$ is not in a
   ramified cycle of $q$, thus  $a_1 \not = 0$ in this case.
   Then by induction we find that
  \[ q^{\circ (\ell + rk)}(x) = c_0 + a_1^k b_e (x-\lambda)^e + \text{higher
    order terms in $(x-\lambda)$}, \]
  so $v_\lambda(q^{\circ (\ell + rk)}(x)-c) = e$ for all $k.$

Suppose now that $c(x)$ is not a constant.  We may suppose that there
are infinitely many $n$ such that $v_\lambda(q^{\circ n} (x) - c(x)) >
e$ since otherwise the lemma clearly holds. Note that,
it's possible that $c_0$ is in a ramified cycle of $q.$ In any case, let
$u$ be the smallest integer $y$ such that $a_{y} \ne 0.$

We first assume that $u = 1$. Equivalently, $c_0$ is not in a ramified cycle of $q$.
Then,   we must have $a_1^k b_e = c_e$ for infinitely many $k$.  Since $a_1 b_e \not= 0$, this
means that $a_1$ must be a root of unity.  Suppose that $a_1^s = 1$.
Then we may write
\[ q^{\circ rs}(x) = c_0 + (x-c_0) + \alpha_d (x-c_0)^d + O\left((x-c_0)^{d+1}\right)
  \]
for some $d > 0$ with $\alpha_d \not= 0.$
It follows that for any $k$, we have
\[ q^{\circ rsk}(x) = c_0 + (x-c_0) + k \alpha_d (x-c_0)^d + O\left((x-c_0)^{d+1}\right)
  \]
Now, let $g(x) = \sum_{i=0}^\infty \beta_i (x-\lambda)^i$ be any
nonconstant polynomial in $(x-\lambda)$ such that $\beta_0 = c_0$.
Let $t$ be the smallest positive integer such that $\beta_t \not= 0$.
Then, for any $k$, the
coefficient of $(x-\lambda)^{td}$ in $q^{\circ rsk} \circ g$ is $k
\alpha_d \beta_t^d + \beta_{td}$.  Since $\alpha_d \not= 0$, there are
in particular at most finitely many $k$ such that  the
coefficient of $(x-\lambda)^{td}$ in $q^{\circ rsk} \circ g$ is equal
to $c_{td}$.  Thus, there are at most finitely many $k$ such that
$v_\lambda(q^{\circ rsk} \circ g(x) - c(x)) > td$, and hence
$v_\lambda(q^{\circ rsk} \circ g(x) - c(x))$ is bounded for all $k$.
Applying this to $g = q^{\circ y}$ for $y = \ell, \ell + r, \dots,
\ell + (s-1) r$ completes our proof, since any number of the form
$\ell + kr$ can be written as $y + krs$ for some such $y$.

Assume now that $u > 1.$  Then, by induction
\[ q^{\circ \ell+ rk}(x) = c_0 + a_u^{(u^k-1)/(u-1)} b_e^{u^k}(x-\lambda)^{e u^k} +  O\left((x-\lambda)^{e u^k + 1}\right). 
\]
So, $v_\lambda(q^{\circ n}(x) - c(x)) \le \deg c$ for  all sufficiently large $k$. Hence,
$v_\lambda(q^{\circ n}(x) - c(x))$ is bounded above by a constant depending on $\lambda$ and $q$ only.

\end{proof}

\begin{remark}
We note that in Lemma~\ref{close}, if $v_\lambda(q^{\circ n}(x) - c(x)) > 0$ then
the integer $n$ is in a congruence class $\ell + r  \bN$ for some
positive integer $r.$ In fact, $r$ is the least period of $c_0 = c(\lambda)$ under the
action of $q.$
\end{remark}

\begin{proof}[Proof of Theorem~\ref{gcdthm}]
We may assume without loss of generality that $c$ is not in a ramified
cycle of $f$.
By Theorem \ref{main thm}, there are at most finitely many $\lambda$
such that $(x-\lambda)$ divides $\gcd (f^{\circ m}(x) - c(x), g^{\circ
  n}(x) - c(x))$ for some $m,n$.  Let $\cS$ be the set of all such
$\lambda$.  By Lemma \ref{close}, there is an $M_\lambda$ such that
$v_\lambda(q^{\circ n}(x) - c(x)) \leq M_{\lambda, q}$ for all $n$,
since $c$ is not a compositional power of $f$.  Then, if
\[ h(x) = \prod_{\lambda_\in \cS} (x-\lambda)^{M_\lambda}, \]
we see that
\[ \gcd(f^{\circ m}(x) - c(x), g^{\circ n}(x) - c(x)) \mid h(x) \]
for all $m,n$, as desired.

\end{proof}

\section{Further directions}\label{sec:question}

Many of the techniques here may work more generally.  We close with
several questions.

Silverman \cite{Sil04} showed that the characteristic $p$ function
field analog of the theorem of Bugeaud-Corvaja-Zannier theorem is not
true; in particular, one can find  multiplicatively
independent polynomials $a, b \in \bF_q[x]$ (where $\bF_q$ is as usual the finite
field with $q$ elements) and an $\epsilon > 0$ such that
$\deg \left(\gcd (a^n - 1, b^n -1)\right) > \epsilon n$ for infinitely many $n$.
Similarly, we suspect that that one can find compositionally
independent polynomials $f,g \in \bF_q[x]$, an $\epsilon > 0$, and a
$c(x) \in \bF_q[x]$ that is not a compositional power of $f$ or $g$
such that
$\deg \left(\gcd (f^{\circ n}(x) - c(x), g^{\circ n}(x) -c(x))\right) > \epsilon
n$ for infinitely many $n$.  On the other hand, on might ask the
following question in characteristic $p$.

\begin{question}\label{charp}
Let $F = \bF_q[T]$ be the polynomial ring in one variable over the
finite field with $q$ elements.  Let $f$ and $g$ be two
compositionally independent nonisotrivial polynomials in $F[x]$, and let $c \in
F[x]$.  Is it true that there are at most finitely many $\lambda \in
{\overline F}$ such that there is an $n$ for which $(x-\lambda)$
divides $\gcd (f^{\circ n}(x) - c(x), g^{\circ n}(x) - c(x))$?  Given
an $\epsilon > 0$ and assuming that $c(x)$ is not in a ramified cycle
of $f$ and $g$, is it even true that
\[\deg\left( \gcd (f^{\circ n}(x) - c(x), g^{\circ n}(x) - c(x))\right) < \epsilon n \]
for all but finitely many $n$?
\end{question}

We might also ask for characteristic 0 results in more general
settings.

\begin{question}\label{dim1}
  Let $\phi_1, \phi_2: \bP^1_\bC \lra \bP^1_\bC$ be two nonconstant,
  compositionally independent morphisms.  Let
  $c: \bP^1_\bC \lra \bP^1_\bC$ be any morphism.  It is true that
  there must be at most finitely many $\lambda \in \bC$ such that
  $\phi_1^{\circ n}(\lambda) = \phi_2^{\circ n}(\lambda) = c(\lambda)$?
\end{question}

We should note that the counterexamples to the dynamical Mordell-Lang
conjecture given in \cite{GTZ} do not yield counterexamples here in an
obvious way, since the Latt\'es maps given there commute with each
other and hence they are not compositionally independent.

For more general varieties, we ask the following.

\begin{question}
Let $V$ be a variety defined over $\bC$ and let $\phi_1, \phi_2: V \lra V$ be two
dominant compositionally independent morphisms.  Let $c: V \lra V$
be any morphism.   Is it true that the set of $\lambda \in V(\bC)$
such that $\phi_1^{\circ n}(\lambda) = \phi_2^{\circ n}(\lambda) = c(\lambda)$ must be
contained in a proper Zariski closed subset of $V$?
\end{question}

In the case where $V$ is projective and some iterates of $\phi_1$ and $\phi_2$ extend to
maps on projective space of degree greater than one (the case where
$\phi_1$ and $\phi_2$ are ``polarizable'' in the language of Zhang
\cite{ZhangLec}), it may be possible, using higher dimensional results such as those
of \cite{Yuan, Gubler, YZ}, to show that $h_{\phi_1} = h_{\phi_2}$
whenever the $\lambda$ such that $\phi_1^n(\lambda) =
\phi_2^n(\lambda) = c(\lambda)$ are Zariski dense.
On the other hand, that may not imply a compositional dependence between
$\phi_1$ and $\phi_2$.  One natural place to look for counterexamples might be abelian
varieties with quaternion endomorphism rings.

One might also ask for results for families of maps; for example, one
might consider polynomials with coefficients in $\bC[t]$ rather than
$\bC$.  The notions of compositional dependency that arise in that
context (see \cite[Theorem 1.2]{BD2}, for example) may be a bit different from the
notion that we use in this paper, and thus, we will refrain from
asking any precise questions here.

Finally, it is natural to ask for a result along the lines of
\cite{BCZ} where one considers iterates of integers under
polynomial maps rather than simply powers of integers
More precisely, one might hope that  $a,b \in \bZ$, two polynomials $f,g \in \bZ[x]$ of degree
$d > 1$, and an $\epsilon > 0$, the inequality
\[ \gcd(f^{\circ n}(a), g^{\circ n}(b)) < \epsilon d^n \]
should hold for all but at most finitely many $n$, given reasonable
conditions on $f$, $g$, $a$, and $b$.
Huang \cite{Keping} has shown that such an inequality must indeed hold for all
sufficiently large $n$ whenever the sequence
$(f^{\circ n}(a), g^{\circ n}(b))_n$ is Zariski dense in $\bA^2$ if
one assumes Vojta's conjecture for heights with respect to canonical
divisors on surfaces (see \cite[Conjecture 3.4.3]{Vojta}).  The proof
uses Silverman's ideas from \cite{SilverGCD}, which relate the
original results of \cite{BCZ} with Vojta's conjecture.

\bibliographystyle{amsalpha}
\def\cprime{$'$} \def\cprime{$'$} \def\cprime{$'$} \def\cprime{$'$}
\providecommand{\bysame}{\leavevmode\hbox to3em{\hrulefill}\thinspace}
\providecommand{\MR}{\relax\ifhmode\unskip\space\fi MR }
\providecommand{\MRhref}[2]{%
  \href{http://www.ams.org/mathscinet-getitem?mr=#1}{#2}
}
\providecommand{\href}[2]{#2}

\end{document}